\documentclass[11pt]{amsart}%
\usepackage{amsmath}
\usepackage{amsfonts}
\usepackage{amssymb}
\usepackage{graphicx}%
\newtheorem{theorem}{Theorem}

\newtheorem{corollary}[theorem]{Corollary}

\newtheorem{lemma}[theorem]{Lemma}

\newtheorem{proposition}[theorem]{Proposition}

\begin{document}
\title{Semilattice Structures of Spreading Models}

\begin{abstract}
Given a Banach space $X$, denote by $SP_{w}(X)$ the set of equivalence classes
of spreading models of $X$ generated by normalized weakly null sequences in
$X$. It is known that $SP_{w}(X)$ is a semilattice, i.e., it is a partially
ordered set in which every pair of elements has a least upper bound. We show
that every countable semilattice that does not contain an infinite increasing
sequence is order isomorphic to $SP_{w}(X)$ for some separable Banach space
$X$.

\end{abstract}
\author[D. H. Leung]{Denny H. Leung}
\address{Department of Mathematics, National University of Singapore, 2 Science Drive
2, Singapore 117543}
\email{matlhh@nus.edu.sg}
\author[W.-K. Tang]{Wee-Kee Tang}
\address{Mathematics and Mathematics Education, National Institute of Education \\
Nanyang Technological University, 1 Nanyang Walk, Singapore 637616}
\email{weekee.tang@nie.edu.sg}
\thanks{Research of the first author was partially supported AcRF project no.\ R-146-000-086-112}
\maketitle

Given a normalized basic sequence $( y_{i}) $ in a Banach space and
$\varepsilon_{n}\searrow0$, using Ramsey's Theorem, one can find a subsequence
$( x_{i}) $ and a normalized basic sequence $( \tilde{x}_{i}) $ such that for
all $n\in\mathbb{N}$ and $( a_{i}) _{i=1}^{n}\subseteq[ -1,1] ,$
\[
\vert\Vert\sum a_{i}x_{k_{i}}\Vert-\Vert\sum a_{i}\tilde{x}_{i}\Vert
\vert<\varepsilon_{n}%
\]
for all $n\leq k_{1}<\dots<k_{n}.$ The sequence $( \tilde{x}_{i}) $ is called
a spreading model of $( x_{i}) .$ It is well-known that if $( x_{i}) $ is in
addition weakly null, then $( \tilde{x}_{i}) $ is 1-spreading and suppression
1-unconditional. See \cite{BS, GB} for more about spreading models. A
spreading model $( \tilde{x}_{i}) $ is said to ($C$-) \emph{dominate} another
spreading model $( \tilde{y}_{i}) $ if there is a $C<\infty$ such that for all
$( a_{i}) \subseteq\mathbb{R}$,%
\[
\Vert\sum a_{i}\tilde{y}_{i}\Vert\leq C\Vert\sum a_{i}\tilde{x}_{i}\Vert.
\]
The spreading models $( \tilde{x}_{i}) $ and $( \tilde{y}_{i}) $ are said to
be \emph{equivalent} if they dominate each other. Let $[ ( \tilde{x}_{i}) ] $
denote the class of all spreading models which are equivalent to $( \tilde
{x}_{i}) .$ Let $SP_{w}( X) $ denote the set of all $[ ( \tilde{x}_{i}) ] $
generated by normalized weakly null sequences in $X.$ If $[ ( \tilde{x}_{i}) ]
$, $[ ( \tilde{y}_{i}) ] \in SP_{w}( X) ,$ we write $[ ( \tilde{x}_{i}) ]
\leq[ ( \tilde{y}_{i}) ] $ if $( \tilde{y}_{i}) $ dominates $( \tilde{x}_{i})
.$ $( SP_{w}( X) ,\leq) $ is a partially ordered set. The paper \cite{AOST}
initiated the study of the order structures of $SP_{w}( X) .$ It was
established that every countable subset of $(SP_{w}(X),\leq)$ admits an upper
bound (\cite[Proposition 3.2]{AOST}). Moreover, from the proof of this result,
it follows that every pair of elements in $(SP_{w}(X),\leq)$ has a least upper
bound. In other words, $(SP_{w}(X),\leq)$ is a \emph{semilattice}. In
\cite{S}, it was shown that if $SP_{w}( X) $ is countable, then it cannot
admit a strictly increasing infinite sequence $( \tilde{x}_{i}^{1}) <(
\tilde{x}_{i}^{2}) <\cdots$. In \cite{DOS}, two methods of construction,
utilizing Lorentz sequence spaces and Orlicz sequence spaces respectively,
were used to produce Banach spaces $X$ so that $SP_{w}(X)$ has certain
prescribed order structures. In the present paper, building on the techniques
employed in \cite[\S 2]{DOS}, we show that every countable semilattice that
has no infinite increasing sequence is order isomorphic to $SP_{w}(X)$ for
some Banach space $X$. This gives an affirmative answer to Problem 1.15 in
\cite{DOS}. (See, however, the remark at the end of the paper.)

\section{A Representation Theorem for Semilattices}

Any collection of subsets of a set $V$ that is closed under the taking of
finite unions is a semilattice under the order of set inclusion. In this
section, we show that any countable semilattice that does not admit an
infinite increasing sequence may be represented in such a way using a
countable set $V$. The result may be of independent interest.

\begin{theorem}
\label{T1}Let $L$ be a countable semilattice with no infinite increasing
sequences. Then there exist a countable set $V$ and an injective map
$T:L\rightarrow2^{V}\smallsetminus\left\{  \emptyset\right\}  $ that preserves
the semilattice structure of $L,$ i.e., $T\left(  x\vee y\right)  =T\left(
x\right)  \cup T\left(  y\right)  $ for all $x,y\in L.$
\end{theorem}

Suppose that $L$ is a semilattice that satisfies the hypothesis of Theorem
\ref{T1}. Note that every nonempty subset of $L$ has at least one maximal
element; for otherwise, it will admit an infinite increasing sequence. Set
$L_{0}=L$. If $L_{\alpha}$ is defined for some countable ordinal $\alpha$ and
$L_{\alpha}\neq\emptyset$, let $L_{\alpha+1}=L_{\alpha}\smallsetminus\left\{
\text{maximal elements in }L_{\alpha}\right\}  .$ If $\alpha$ is a countable
limit ordinal such that $L_{\alpha^{\prime}} \neq\emptyset$ for all
$\alpha^{\prime}< \alpha$, let $L_{\alpha}= \cap_{\alpha^{\prime}<\alpha
}L_{\alpha^{\prime}}.$ Since $\left(  L_{\alpha}\right)  $ is a strictly
decreasing transfinite sequence of subsets of the countable set $L$,
$L_{\alpha}= \emptyset$ for some countable ordinal $\alpha$. Let $\alpha_{0}$
be the smallest ordinal such that $L_{\alpha_{0}} = \emptyset$. Enumerate $L$
as a transfinite sequence $\left(  e_{\beta}\right)  _{\beta<\beta_{0}}$ so
that if $e_{\beta_{1}}\in L_{\alpha_{1}}\smallsetminus L_{\alpha_{1}+1}$ and
$e_{\beta_{2}}\in L_{\alpha_{2}}\smallsetminus L_{\alpha_{2}+1}$ for some
$\alpha_{1}<\alpha_{2} < \alpha_{0}$, then $\beta_{1}<\beta_{2}$. If $1
\leq\beta\leq\beta_{0},$ let $U_{\beta}=\left\{  e_{\beta^{\prime}}%
:\beta^{\prime}<\beta\right\}  .$ Note that $L=U_{\beta_{0}}.$

\begin{lemma}
\label{L1}

\begin{enumerate}
\item[(a)] $e_{\beta}$ is a minimal element in $U_{\beta+1}.$

\item[(b)] If $e_{\beta} = e_{\beta_{1}}\vee e_{\beta_{2}}$ (least upper bound
taken in $L$), then $\beta\leq\min\left\{  \beta_{1},\beta_{2}\right\}  .$

\item[(c)] If $e_{\beta_{1}}, e_{\beta_{2}} \in U_{\beta}$, then $e_{\beta
_{1}}\vee e_{\beta_{2}}$ belongs to $U_{\beta}$.

\end{enumerate}
\end{lemma}

\begin{proof}
(a) Suppose on the contrary that $e_{\beta}$ is not a minimal element in
$U_{\beta+1}$. Then there exists $e_{\beta^{\prime}}\in U_{\beta+1}$ with
$e_{\beta^{\prime}}<e_{\beta}.$ It follows from the definition of $U_{\beta
+1}$ that $\beta^{\prime}< \beta.$ If $e_{\beta}\in L_{\alpha}\smallsetminus
L_{\alpha+1}$ and $e_{\beta^{\prime}}\in L_{\alpha^{\prime}}\smallsetminus
L_{\alpha^{\prime}+1},$ then $\alpha^{\prime}\leq\alpha$ and hence $L_{\alpha
}\subseteq L_{\alpha^{\prime}}.$ Since $e_{\beta}, e_{\beta^{\prime}} \in
L_{\alpha^{\prime}}$ and $e_{\beta^{\prime}} < e_{\beta}$, $e_{\beta^{\prime}%
}$ is not maximal in $L_{\alpha^{\prime}}$. Thus $e_{\beta^{\prime}} \in
L_{\alpha^{\prime}+ 1}$, a contradiction.

(b) Suppose that $\beta_{1}<\beta$. Then $e_{\beta_{1}}\in U_{\beta+1}$ and
$e_{\beta_{1}} < e_{\beta}$, contrary to the minimality of $e_{\beta}$ in
$U_{\beta+1}.$ Similarly, $\beta_{2}\geq\beta.$

(c) Follows immediately from (b).
\end{proof}

If $1\leq\beta<\omega_{1},$ write $\beta=\gamma+n,$ where $\gamma$ is a limit
ordinal, $n<\omega,$ and let $V_{\beta}$ denote the ordinal interval
$[0,\gamma+2n).$ We define a family of maps $T_{\beta}:U_{\beta}%
\rightarrow2^{V_{\beta}}\smallsetminus\left\{  \emptyset\right\}  ,$
$1\leq\beta\leq\beta_{0}$, inductively so that $T=T_{\beta_{0}}$ is the map
sought for in Theorem \ref{T1}. Let $T_{1}:U_{1}=\left\{  e_{0}\right\}
\rightarrow2^{V_{1}}\smallsetminus\left\{  \emptyset\right\}  $ be defined by
$T_{1}\left(  e_{0}\right)  =\left\{  0,1\right\}  .$ If $T_{\beta}$ has been
defined, $1\leq\beta<\beta_{0}$, let
\[
T_{\beta+1}\left(  x\right)  =\left\{
\begin{array}
[c]{ccc}%
T_{\beta}\left(  x\right)  \cup\left\{  \gamma+2n,\gamma+2n+1\right\}  &
\text{if} & x\in U_{\beta+1}\setminus\{e_{\beta}\},\\
\bigcap\limits_{e_{\beta}<z\in U_{\beta}}T_{\beta}\left(  z\right)
\cup\left\{  \gamma+2n+1\right\}  & \text{if} & x=e_{\beta}.
\end{array}
\right.
\]
When $\beta\leq\beta_{0}$ is a limit ordinal and$\ e_{\beta^{\prime}}\in
U_{\beta}$, let $T_{\beta}( e_{\beta^{\prime}}) =\cup_{\beta^{\prime}%
<\xi<\beta}T_{\xi}( e_{\beta^{\prime}})$. The next result, which shows the
compatibility of the definitions of $T_{\beta}$ for different $\beta$'s, is
the key to the subsequent arguments.

\begin{lemma}
\label{L3}If $1\leq\beta_{1}<\beta_{2}\leq\beta_{0}$ and $\beta_{i}=\gamma
_{i}+n_{i},$ $i=1,2$, then
\[
T_{\beta_{2}}( e_{\beta_{1}}) =T_{\beta_{1}+1}( e_{\beta_{1}}) \cup
\lbrack\gamma_{1}+2n_{1}+2,\gamma_{2}+2n_{2}).
\]

\end{lemma}

\begin{proof}
If $\beta_{2}=\beta_{1}+1,$ the assertion holds clearly. Suppose that the
assertion holds for some $\beta_{2}>\beta_{1}.$ By the definition of
$T_{\beta_{2}+1}$,%
\begin{align*}
T_{\beta_{2}+1}\left(  e_{\beta_{1}}\right)   &  =T_{\beta_{2}}\left(
e_{\beta_{1}}\right)  \cup\left\{  \gamma_{2}+2n_{2},\gamma_{2}+2n_{2}%
+1\right\} \\
&  =T_{\beta_{1}+1}\left(  e_{\beta_{1}}\right)  \cup\lbrack\gamma_{1}%
+2n_{1}+2,\gamma_{2}+2n_{2})\\
&  \quad\quad\cup\left\{  \gamma_{2}+2n_{2},\gamma_{2}+2n_{2}+1\right\} \\
&  =T_{\beta_{1}+1}\left(  e_{\beta_{1}}\right)  \cup\lbrack\gamma_{1}%
+2n_{1}+2,\gamma_{2}+2n_{2}+2).
\end{align*}

Suppose that $\beta_{2}\leq\beta_{0}$ is a limit ordinal and the assertion
holds for all $\beta_{1}<\xi<\beta_{2}.$ For such $\xi$, let $\xi=\gamma_{\xi
}+n_{\xi}.$ By the inductive hypothesis,
\[
T_{\xi}(e_{\beta_{1}})=T_{\beta_{1}+1}(e_{\beta_{1}})\cup\lbrack\gamma
_{1}+2n_{1}+2,\gamma_{\xi}+2n_{\xi}).
\]
Since $\beta_{2}$ is a limit ordinal, we have
\begin{align*}
T_{\beta_{2}}(e_{\beta_{1}})  &  =\cup_{\beta_{1}<\xi<\beta_{2}}T_{\xi
}(e_{\beta_{1}})\\
&  =T_{\beta_{1}+1}(e_{\beta_{1}})\cup\lbrack\gamma_{1}+2n_{1}+2,\beta_{2})\\
&  =T_{\beta_{1}+1}\left(  e_{\beta_{1}}\right)  \cup\lbrack\gamma_{1}%
+2n_{1}+2,\gamma_{2}+2n_{2}),
\end{align*}
as required. (Note that $n_{2}=0$ since $\beta_{2}$ is a limit ordinal).
\end{proof}

\begin{lemma}
\label{C1} The map $T_{\beta}:U_{\beta}\rightarrow2^{V_{\beta}}\smallsetminus
\left\{  \emptyset\right\}  $ is injective if $1\leq\beta\leq\beta_{0}$.
\end{lemma}

\begin{proof}
Suppose that $e_{\beta_{1}}$ and $e_{\beta_{2}}$ are distinct elements in
$U_{\beta}$, with $\beta_{1}<\beta_{2}<\beta.$ Write $\beta_{2}=\gamma
_{2}+n_{2}$. It follows from Lemma \ref{L3} that $\gamma_{2}+2n_{2}\in
T_{\beta}(e_{\beta_{1}})\smallsetminus T_{\beta}(e_{\beta_{2}})$.
\end{proof}

\begin{proposition}
\label{P1} If $1\leq\beta\leq\beta_{0}$, then $T_{\beta}\left(  x\vee
y\right)  =T_{\beta}\left(  x\right)  \cup T_{\beta}\left(  y\right)  $ for
all $x,y\in U_{\beta}$. In particular, $T_{\beta}\left(  x\right)  \subseteq
T_{\beta}\left(  y\right)  $ if $x\leq y,$.
\end{proposition}

\begin{proof}
The second statement follows easily from the first. We prove the first
statement by induction on $\beta.$ The result is clear if $\beta=1$. Suppose
that the assertion is true for some $\beta,$ $1\leq\beta<\beta_{0}$. Let
$x=e_{\beta_{1}},y=e_{\beta_{2}}\in U_{\beta+1}.$ We may assume that
$\beta_{1}<\beta_{2}<\beta+1.$ Write $\beta=\gamma+n,$ and $\beta_{i}%
=\gamma_{i}+n_{i},$ $i=1,2$, and consider two cases.\newline

\noindent\underline{Case 1}. $\beta_{1}<\beta_{2}<\beta.$ \newline By Lemma
\ref{L3} and the inductive hypothesis,
\begin{align*}
T_{\beta+1}\left(  e_{\beta_{1}}\right)   &  \cup T_{\beta+1}\left(
e_{\beta_{2}}\right) \\
&  =T_{\beta_{1}+1}\left(  e_{\beta_{1}}\right)  \cup\lbrack\gamma_{1}%
+2n_{1}+2,\gamma+2n)\cup\\
&  \quad\quad\cup T_{\beta_{2}+1}\left(  e_{\beta_{2}}\right)  \cup
\lbrack\gamma_{2}+2n_{2}+2,\gamma+2n)\cup\{\gamma+2n,\gamma+2n+1\}\\
&  =T_{\beta}\left(  e_{\beta_{1}}\right)  \cup T_{\beta}\left(  e_{\beta_{2}%
}\right)  \cup\{\gamma+2n,\gamma+2n+1\}\\
&  =T_{\beta}\left(  e_{\beta_{1}}\vee e_{\beta_{2}}\right)  \cup
\{\gamma+2n,\gamma+2n+1\}\\
&  =T_{\beta+1}(e_{\beta_{1}}\vee e_{\beta_{2}}),
\end{align*}
by definition of $T_{\beta+1}$, since $e_{\beta_{1}}\vee e_{\beta_{2}}\neq
e_{\beta}$ by part (b) of Lemma \ref{L1}.\newline

\noindent\underline{Case 2}. $\beta_{1}<\beta_{2}=\beta.$\newline In this
case,
\[
T_{\beta+1}\left(  x\right)  \cup T_{\beta+1}\left(  y\right)  =\bigcap
\limits_{e_{\beta}<z\in U_{\beta}}\left[  T_{\beta}\left(  x\right)  \cup
T_{\beta}\left(  z\right)  \right]  \cup\left\{  \gamma+2n,\gamma
+2n+1\right\}  .
\]
Note that by part (b) of Lemma \ref{L1}, $x\vee e_{\beta}=e_{\xi}$ for some
$\xi\leq\beta_{1}$. Hence, $x\vee e_{\beta}\in U_{\beta+1}\smallsetminus
\{e_{\beta}\}=U_{\beta}$. Thus, it suffices to show that
\[
\bigcap\limits_{e_{\beta}<z\in U_{\beta}}\left[  T_{\beta}\left(  x\right)
\cup T_{\beta}\left(  z\right)  \right]  =T_{\beta}\left(  x\vee e_{\beta
}\right)  =T_{\beta}\left(  x\vee y\right)  .
\]
Since $e_{\beta}<x\vee e_{\beta}\in U_{\beta}$, $\bigcap\limits_{e_{\beta
}<z\in U_{\beta}}T_{\beta}\left(  z\right)  \subseteq T_{\beta}\left(  x\vee
e_{\beta}\right)  .$ By the inductive hypothesis, $T_{\beta}\left(  x\right)
\subseteq T_{\beta}\left(  x\vee e_{\beta}\right)  .$ It follows that
$\bigcap\limits_{e_{\beta}<z\in U_{\beta}}\left[  T_{\beta}\left(  x\right)
\cup T_{\beta}\left(  z\right)  \right]  \subseteq T_{\beta}\left(  x\vee
e_{\beta}\right)  .$ On the other hand, if $e_{\beta}<z\in U_{\beta},$ then
$x\vee e_{\beta}\leq x\vee z\in U_{\beta}.$ By the inductive hypothesis,
$T_{\beta}\left(  x\vee e_{\beta}\right)  \subseteq T_{\beta}\left(  x\vee
z\right)  =T_{\beta}\left(  x\right)  \cup T_{\beta}\left(  z\right)  .$
Therefore, $T_{\beta}\left(  x\vee e_{\beta}\right)  \subseteq\bigcap
\limits_{e_{\beta}<z\in U_{\beta}}[ T_{\beta}\left(  x\right)  \cup T_{\beta
}\left(  z\right)  ] $.

Suppose that $\beta$ is a limit ordinal and the Proposition holds for all
$\beta^{\prime}<\beta.$ Let $x,y\in U_{\beta}.$ We may assume that
$x=e_{\beta_{1}}$ and $y=e_{\beta_{2}}$ for some $\beta_{1}<\beta_{2}<\beta.$
Let $\beta_{i}=\gamma_{i}+n_{i}$, $i=1,2$. Using Lemma \ref{L3} and the
inductive hypothesis,
\begin{align*}
T_{\beta}( x) \cup T_{\beta}( y)  &  =T_{\beta_{1}+1}( e_{\beta_{1}}) \cup
T_{\beta_{2}+1}( e_{\beta_{2}}) \cup\lbrack\gamma_{1}+2n_{1}+2,\beta)\\
&  =T_{\beta_{2}+1}( e_{\beta_{1}}) \cup T_{\beta_{2}+1}( e_{\beta_{2}})
\cup\lbrack\gamma_{2}+2n_{2}+2,\beta)\\
&  =T_{\beta_{2}+1}( e_{\beta_{1}}\vee e_{\beta_{2}}) \cup\lbrack\gamma
_{2}+2n_{2}+2,\beta).
\end{align*}
By (b) of Lemma \ref{L1}, $e_{\beta_{1}}\vee e_{\beta_{2}}=e_{\eta}$ for some
$\eta\leq\beta_{1}$. By Lemma \ref{L3},
\begin{align*}
T_{{\beta_{2}}+1}(e_{\beta_{1}}\vee e_{\beta_{2}})  &  =T_{\eta+1}(e_{\eta
})\cup\lbrack\gamma_{\eta}+2n_{\eta}+2,\gamma_{2}+2n_{2}+2),\\
\text{{and }}T_{{\beta}+1}(e_{\beta_{1}}\vee e_{\beta_{2}})  &  =T_{\eta
+1}(e_{\eta})\cup\lbrack\gamma_{\eta}+2n_{\eta}+2,\beta),
\end{align*}
where $\eta=\gamma_{\eta}+n_{\eta}$. Combining the three preceding equations
gives $T_{\beta}(x)\cup T_{\beta}(y)=T_{\beta}(x\vee y)$.
\end{proof}

\begin{proof}
[Proof of Theorem \ref{T1}]Since $L=U_{\beta_{0}}$, Theorem \ref{T1} follows
immediately from Lemma \ref{C1} and Proposition \ref{P1} by taking
$\beta=\beta_{0}$ in each instance.
\end{proof}

\section{Good Lorentz Functions}

A \emph{Lorentz sequence} is a non-increasing sequence $(w(n))_{n=1}^{\infty}$
of positive numbers such that $w(1)=1$, $\lim_{n}w(n)=0$ and $\sum
_{n=1}^{\infty}w(n)=\infty$. A Lorentz sequence is $C$%
-\emph{submultiplicative} if $S(mn)\leq CS(m)S(n)$ for all $m,n\in\mathbb{N}$,
where $S(n)=\sum_{k=1}^{n}w(k).$ In \cite[\S 2]{DOS}, an infinite sequence of
$1$-submultiplicative Lorentz sequences is constructed so that the maxima of
any two incomparable \emph{finite} subsets are incomparable (see
\cite[Proposition 2.6]{DOS}). For our purpose, we require an infinite sequence
of $C$-submultiplicative Lorentz sequences so that the supremum of any (finite
or infinite) subset remains a $C$-submultiplicative Lorentz function, and that
the suprema of any two incomparable (finite or infinite) subsets are
incomparable (Proposition \ref{A4}). This is done by tweaking the arguments in
\cite[\S 2]{DOS}. Following \cite{DOS}, we will find it more convenient to
work with functions defined on real intervals. If $2\leq N<\infty$, a
\emph{good Lorentz function} (GLF) on $(0,N]$ is a function
$w:(0,N]\rightarrow(0,\infty)$ such that

\begin{enumerate}
\item $w(x) = 1$, $x \in(0,2]$,

\item $w$ is nonincreasing, and

\item If $1\leq x,y\leq xy\leq N$, then $\int_{0}^{xy}w\leq\int_{0}^{x}%
w\cdot\int_{0}^{y}w$.
\end{enumerate}

A \emph{GLF on} $(0,\infty)$ (or simply a GLF) is a function $w:\left(
0,\infty\right)  \rightarrow\left(  0,\infty\right)  $ such that $w_{|(0,N]}$
is a GLF on $(0,N]$ for any $N\geq2$, $\lim_{x\rightarrow\infty}w(x)=0$ and
$\int_{0}^{\infty}w=\infty.$ It is an easy exercise to verify that if $w$ is a
GLF, then $(w(n))_{n=1}^{\infty}$ is a $4$-submultiplicative Lorentz sequence.

If $(u_{i})$ is a finite or infinite sequence of real-valued functions with
pairwise disjoint domains, let $\oplus_{i} u_{i}$ denote the set theoretic
union. The constant $1$ function with domain $I$ is denoted by $1_{I}$. We now
recall the relevant facts from \cite{DOS}. Note that the quantity $S(x)$ there
corresponds to $\int_{0}^{x}w$ in our notation.

\begin{lemma}
{\cite[Lemma 2.2]{DOS}}\label{A1} Let $w$ be a GLF on $(0,N]$, $N \geq2$. Then
there exists $\varepsilon_{0}>0$ such that for all $\varepsilon<\varepsilon
_{0},$ $w\oplus\varepsilon1_{(N,N^{2}]}$ is a GLF on $(0,N^{2}].$
\end{lemma}

Repeated applications of Lemma \ref{A1} yield

\begin{lemma}
\label{A2}Let $G$ be a finite set of GLF's on $(0,N]$, $N\geq2$. For any
$N^{\prime}>N$ and any $\varepsilon>0$, there is a function $v:(N,N^{\prime
}]\rightarrow(0,\infty)$ such that $w\oplus v$ is a GLF on $(0,N^{\prime}]$
for all $w\in G$, $v(x)\leq\varepsilon$, $x\in(N,N^{\prime}]$, and $\int
_{N}^{N^{\prime}}v<\varepsilon$.
\end{lemma}

On the other hand, the proof of \cite[Lemma 2.4]{DOS} allows us to obtain GLF
extensions with large total weight.

\begin{lemma}
Let $G$ be a finite set of GLF's on $(0,N]$, $N\geq2$ and set $K=\min_{w\in
G}\int_{0}^{N}w$. For any $\varepsilon>0$, there is a function $v:(N,N^{\prime
}]\rightarrow(0,\infty)$, $N^{\prime}>N$, such that

\begin{enumerate}
\item For all $w\in G$, $w\oplus v$ is a GLF on $(0,N^{\prime}]$,

\item $v(x)\leq\varepsilon$, $x\in(N,N^{\prime}]$,

\item $\int_{N}^{N^{\prime}}v\geq\frac{K}{2}$.
\end{enumerate}
\end{lemma}

We may repeat the preceding lemma to obtain

\begin{lemma}
\label{A3}Let $G$ be a finite set of GLF's on $(0,N]$, $N\geq2$. For any
$K<\infty$ and any $\varepsilon>0$, there is a function $v:(N,N^{\prime
}]\rightarrow(0,\infty)$, $N^{\prime}>N$, such that

\begin{enumerate}
\item For all $w\in G$, $w\oplus v$ is a GLF on $(0,N^{\prime}]$,

\item $v(x)\leq\varepsilon$, $x\in(N,N^{\prime}]$,

\item $\int_{N}^{N^{\prime}}v\geq K$.
\end{enumerate}
\end{lemma}

\begin{proposition}
\label{A4}There exists an infinite sequence $\left(  w_{p}\right)  ^{\infty
}_{p=1} $ of GLF's on $\left(  0,\infty\right)  $ such that for every nonempty
$M\subseteq\mathbb{N}$ and every $p^{\prime}\notin M,$

\begin{enumerate}
\item $w_{M}=\sup_{p\in M}w_{p}$ is a GLF on $\left(  0,\infty\right)  $,

\item
\[
\sup_{n}\frac{\int_{0}^{n}w_{p^{\prime}}}{\int_{0}^{n}w_{M}}=\infty.
\]

\end{enumerate}
\end{proposition}

\begin{proof}
The desired family of incomparable GLF's is constructed by defining its
elements inductively on successive intervals. On each of the segments, each of
the $w_{p}$'s is chosen to be either ``high" or ``low".

Let $((p_{i},q_{i}))_{i=1}^{\infty}$ be an enumeration of $\{(p,q):p<q,p,q\in
{\mathbb{N}}\}$ and fix a positive sequence $(\varepsilon_{i})$ decreasing to
$0$. For all $p\in{\mathbb{N}}$, define $w_{p}^{0}:(0,2]\rightarrow(0,\infty)$
by $w_{p}^{0}(x)=1$. Set $G_{0}=\{w_{p}^{0}: p \in{\mathbb{N}}\}$.

Assume that for some $i \in{\mathbb{N}}$, functions $w_{p}^{j}:(N_{j-1}%
,N_{j}]\rightarrow(0,\infty)$, $0\leq j<i$ ($N_{-1}=0$, $N_{0}=2$),
$p\in{\mathbb{N}}$, have been defined so that $G_{i-1}=\{w_{r_{0}}^{0}%
\oplus\cdots\oplus w_{r_{i-1}}^{i-1}:r_{0},\dots,r_{i-1}\in{\mathbb{N}}\}$ is
a finite set of GLF's on $(0,N_{i-1}]$ and that $\{w_{p}^{j}:p\in{\mathbb{N}%
}\}$ is a totally ordered set of functions (in the pointwise order) for each
$j \in[0,i)$. Set $K_{i-1}=\int_{0}^{N_{i-1}}\max G_{i-1}$, where by $\max
G_{i-1}$ we mean the pointwise maximum of the set of functions $G_{i-1}$. By
Lemma \ref{A3}, choose a function $w_{p_{i}}^{i}$ on $(N_{i-1},N_{i}]$,
$N_{i}>N_{i-1}$, such that $w\oplus w_{p_{i}}^{i}$ is a GLF on $(0,N_{i}]$ for
all $w\in G_{i-1}$, that $w_{p_{i}}^{i}(x)\leq\varepsilon_{i}$ for all
$x\in(N_{i-1},N_{i}]$ and that $\int_{N_{i-1}}^{N_{i}}w_{p_{i}}^{i}\geq
q_{i}K_{i-1}$. On the other hand, by Lemma \ref{A2}, there exists $v$ on
$(N_{i-1},N_{i}]$ such that $w\oplus v$ is a GLF on $(0,N_{i}]$ for all $w\in
G_{i-1}$, that $v(x)\leq w_{p_{i}}^{i}(N_{i})$ for all $x\in(N_{i-1},N_{i}]$
and that $\int_{N_{i-1}}^{N_{i}}v\leq1$. Define $w_{p}^{i}=v$ for all $p\neq
p_{i}$. Note that $G_{i}=\{w\oplus w_{p}^{i}:w\in G_{i-1},p\in{\mathbb{N}}\}$
is a finite set of GLF's on $(0,N_{i}]$. Obviously, the set $\{w_{p}^{i}%
:p\in{\mathbb{N}}\} = \{w^{i}_{p_{i}}, v\}$ is totally ordered. This completes
the inductive construction. Define $w_{p}=\oplus_{i}w_{p}^{i}$, $p\in
{\mathbb{N}}$. Observe that $K_{0}=2$ and $K_{i}\geq K_{i-1}+q_{i}K_{i-1}%
\geq3K_{i-1}$. Hence $K_{i} \to\infty$. Thus
\[
N_{i} \geq N_{i}-N_{i-1}\geq\int_{N_{i-1}}^{N_{i}}w_{p_{i}}^{i}\geq
q_{i}K_{i-1}\rightarrow\infty.
\]
Hence $w_{p}$ is defined on $(0,\infty)$ for all $p\in\mathbb{N}$. If
$\emptyset\neq M\subseteq{\mathbb{N}}$, let $w_{M}=\sup_{p\in M}w_{p}$. We
claim that $w_{M}$ is a GLF on $(0,\infty)$. By definition, ${w_{M}%
}_{|(0,N_{i}]}\in G_{i}$ for all $i\in{\mathbb{N}}$. Thus $w_{M}$ is a GLF on
$(0,N_{i}]$ for all $i\in{\mathbb{N}}$. Also note that $w_{M}\left(  x\right)
\leq\varepsilon_{i}$ for all $x\in(N_{i-1},N_{i}].$ Therefore, $\lim
_{x\rightarrow\infty}w_{M}\left(  x\right)  =0.$ Furthermore, since $w_{M} =
w^{i}_{p_{i}}$ on $(N_{i-1},N_{i}]$ if $p_{i} \in M$,
\[
\int_{0}^{\infty} w_{M}> \sup_{\{i: p_{i} \in M\}} \int_{N_{i-1}}^{N_{i}}%
w^{i}_{p_{i}}\geq\sup_{\{i: p_{i} \in M\}}q_{i}K_{i-1}.
\]
Because of the enumeration, $p_{i} \in M$ holds for infinitely many $i$. It
follows that $\int_{0}^{\infty} w_{M} = \infty$. This shows that $w_{M}$ is a
GLF on $(0,\infty)$.

Finally, note that for all $i$ such that $p_{i}\notin M$, $\int_{N_{i-1}%
}^{N_{i}}w_{M}\leq1$ by construction. In particular, if $p^{\prime}\notin M$,
then for all $i$ such that $p_{i}=p^{\prime},$%
\begin{align*}
\int_{0}^{N_{i}}w_{M}  &  \leq\int_{0}^{N_{i-1}}w_{M}+\int_{N_{i-1}}^{N_{i}%
}w_{M}\\
&  \leq\int_{0}^{N_{i-1}}\max G_{i-1}+\max_{p\in M}\int_{N_{i-1}}^{N_{i}}w_{p}
\leq K_{i-1}+1.
\end{align*}
On the other hand, for all such $i$,
\[
\int_{0}^{N_{i}}w_{p^{\prime}}\geq\int_{N_{i-1}}^{N_{i}}w_{p^{\prime}}%
=\int_{N_{i-1}}^{N_{i}}w_{p_{i}}^{i}\geq q_{i}K_{i-1}.
\]
Hence%
\[
\sup_{n}\frac{\int_{0}^{n}w_{p^{\prime}}}{\int_{0}^{n}w_{M}}=\infty.
\]

\end{proof}

Given a Lorentz sequence $(w(n))^{\infty}_{n=1}$ and $1 \leq p < \infty$, the
Lorentz sequence space $d(w,p)$ consists of all real sequences $(a_{n})$ such
that $\sum a^{*}_{n}w_{n} < \infty$, where $(a^{*}_{n})$ denotes the
non-increasing rearrangement of $(|a_{n}|)$.

\begin{corollary}
\label{A5}Let $\left(  w_{p}\right)  _{p=1}^{\infty}$ be as above. For every
$M\subseteq\mathbb{N}$, and $p\notin M,$ the unit vector basis of $d\left(
w_{M},1\right)  $ does not dominate that of $d\left(  w_{p},1\right)  $.
\end{corollary}

\begin{proof}
Let $\left(  v_{i}\right)  $ and $\left(  u_{i}\right)  $ denote the
respective unit vector bases of $d\left(  w_{p},1\right)  $ and $d\left(
w_{M},1\right)  .$ According to Proposition \ref{A4}, for any $K < \infty$,
there exists $N\in\mathbb{N}$ such that $\int_{0}^{N+1}w_{p}\geq K\int
_{0}^{N+1}w_{M}.$ Then%
\begin{align*}
\bigl\Vert
{\textstyle\sum\limits_{i=1}^{N}}
v_{i}\bigr\Vert  &  =%
{\textstyle\sum\limits_{i=1}^{N}}
w_{p}( i) \geq\int_{1}^{N+1}w_{p} = \int^{N+1}_{0}w_{p} - 1\\
&  \geq K\int_{0}^{N+1}w_{M} -1 \geq K {\textstyle\sum\limits_{i=1}^{N}}
w_{M}\left(  i\right)  -1 = K\bigl\Vert
{\textstyle\sum\limits_{i=1}^{N}}
u_{i}\bigr\Vert - 1.
\end{align*}
The result follows since $K$ is arbitrary.
\end{proof}

\section{Countable Semilattices of Spreading Models}

In this section, we show that every countable semilattice without an infinite
increasing sequence is order isomorphic to some $SP_{w}\left(  X\right)  $. If
$\left(  x_{i}\right)  $ and $\left(  y_{i}\right)  $ are sequences in the
Banach spaces $X$ and $Y$ respectively, let $\left(  x_{i}\right)
\oplus\left(  y_{i}\right)  $ denote the sequence $\left(  z_{i}\right)
=\left(  x_{i},y_{i}\right)  $ in the direct sum $X\oplus Y.$ The $\ell^{p}%
$-sum of an infinite sequence $(X_{j})$ of Banach spaces is denoted by
$(\sum_{j=1}^{\infty}\oplus X_{j})_{p}$. We omit the easy proof of the next lemma.

\begin{lemma}
\label{C2}Let $w_{1}=\left(  w_{1}\left(  n\right)  \right)  $ and
$w_{2}=\left(  w_{2}\left(  n\right)  \right)  $ be Lorentz sequences. Then
$w=w_{1}\vee w_{2}=\left(  w_{1}\left(  n\right)  \vee w_{2}\left(  n\right)
\right)  $ is a Lorentz sequence. Moreover, if $(u_{n}^{1})$ and $(u_{n}^{2})$
are the respective unit vector bases of $d(w_{1},1)$ and $d(w_{2},1)$, then
$(u_{n}^{1})\oplus(u_{n}^{2})$ is equivalent to $(u_{n})$, the unit vector
basis of $d(w,1)$.
\end{lemma}

\begin{lemma}
[{\cite[Lemma 3.6]{DOS}}]\label{C3A}Let $X=( \sum_{j=1}^{\infty}\oplus
X_{j})_{p}$, where $1\leq p<\infty$ and each $X_{j}$ is an
infinite-dimensional Banach space, and let $(\tilde{x}_{i})$ be a spreading
model generated by a normalized weakly null sequence in $X$. Then there exist
non-negative $(c_{j})_{j=0}^{\infty}$ with $\sum_{j=0}^{\infty}c_{j}^{p}=1$
and normalized spreading models $(\tilde{x}_{i}^{j})_{i}$ in $X_{j}$ generated
by weakly null sequences such that for all scalars $(a_{i}),$
\begin{equation}
\Vert\sum_{i}a_{i}\tilde{x}_{i}\Vert=\bigl[\, \sum_{j=1}^{\infty}\,c_{j}%
^{p}\Vert\sum_{i}a_{i}\tilde{x}_{i}^{j}\Vert^{p}+c_{0}^{p}\sum_{i}|a_{i}%
|^{p}\bigr] ^{1/p}. \label{c0}%
\end{equation}

\end{lemma}

\noindent\textbf{Remark}. If $p=1,$ the final term on the right of equation
(\ref{c0}) may be omitted, i.e., $c_{0} = 0$. In fact, according to the proof
of Lemma \ref{C3A} in \cite[Lemma 3.6]{DOS}, the spreading model $\left(
\tilde{x}_{i}\right)  $ is generated by a weakly null sequence $\left(
x_{i}\right)  $ in $X$ in such a way that $c_{0}=\lim\left\Vert x_{i}%
-P_{i}\left(  x_{i}\right)  \right\Vert ,$ where $P_{i}\left(  x_{i}\right)
=\left(  x_{i}^{1},x_{i}^{2},\cdots,x_{i}^{i},0,0,0,\cdots\right)  .$ However,
since $\ell^{1}$ has the Schur property (weakly null sequences are norm null),
it is easy to see that $\lim\left\Vert x_{i}-P_{i}\left(  x_{i}\right)
\right\Vert = 0$ for any weakly null sequence $(x_{i})$ in $(\sum
_{j=1}^{\infty}\oplus X_{j})_{1}$.\newline

The following is the crucial property of Lorentz sequence spaces that we
require. It can be deduced from the arguments in \cite[\S 4]{ACL}:

\begin{theorem}
\cite{ACL} \label{ThmACL} Let $w=\left(  w\left(  n\right)  \right)  $ be a
$C$-submultiplicative Lorentz sequence and let $(u_{n})$ be the unit vector
basis of $d(w,1)$. For any $\varepsilon>0,$ every normalized block basis in
$d\left(  w,1\right)  $ has a subsequence $\left(  x_{n}\right)  $ such that either

\begin{enumerate}
\item[(a)] $\left(  x_{n}\right)  $ is equivalent to the unit vector basis of
$\ell^{1},$ or

\item[(b)] there exists $c>0$ such that for all $\left(  a_{n}\right)  \in
c_{00},$%
\begin{equation}
c\Vert\sum a_{n}u_{n}\Vert\leq\Vert\sum a_{n}x_{n}\Vert\leq( C+\varepsilon
)\Vert\sum a_{n}u_{n}\Vert. \label{ACLequivalence}%
\end{equation}

\end{enumerate}

In particular, if $\left(  \tilde{x}_{n}\right)  $ is a spreading model
generated by a normalized weakly null sequence, then $\left(  \tilde{x}%
_{n}\right)  $ satisfies $($\ref{ACLequivalence}$)$ in place of $\left(
x_{n}\right)  .$
\end{theorem}

\begin{theorem}
Given a countable semilattice $L$ with no infinite increasing sequence, there
is a Banach space $X_{L}$ such that $SP_{w}\left(  X_{L}\right)  $ is order
isomorphic to $L.$
\end{theorem}

\begin{proof}
By Theorem \ref{T1}, there exists a countable set $V$ and an injective map
$T:L\rightarrow2^{V}\smallsetminus\left\{  \emptyset\right\}  $ such that
$T\left(  e\vee f\right)  =T\left(  e\right)  \cup T\left(  f\right)  $ for
all $e,f\in L.$ Since $V$ is countable, by Proposition \ref{A4} (and Corollary
\ref{A5}), there is a family $\left(  w_{v}\right)  _{v\in V}$ of
$4$-submultiplicative GLF's such that for each non-empty subset $M$ of $V$,
$w_{M}=\sup_{v\in M}w_{v}$ is again a (4-submultiplicative) GLF. Moreover, if
$p\notin M,$ the unit vector basis of $d\left(  w_{M},1\right)  $ does not
dominate that of $d\left(  w_{p},1\right)  $. Set $X_{L}=\left(
{\textstyle\bigoplus_{e\in L}}
d\left(  w_{Te},1\right)  \right)  _{1}. $ For any $e \in L$, let $(u^{e}%
_{i})$ be the unit vector basis of $d(w_{Te},1)$. $(u^{e}_{i})$ may be
regarded in an obvious way as a normalized weakly null sequence in $X_{L}$
which generates a spreading model equivalent to itself. Thus $[(u^{e}_{i})]$,
the equivalence class containing $(u^{e}_{i})$, is an element of $SP_{w}%
(X_{L})$. Define a map $\Theta:L\rightarrow SP_{w}\left(  X_{L}\right)  $ by
$\Theta e=\left[  \left(  u_{i}^{e}\right)  \right]  .$ We will show that
$\Theta$ is a bijection such that $\Theta e_{1}\leq\Theta e_{2}$ if and only
if $e_{1}\leq e_{2}.$ Hence $SP_{w}\left(  X_{L}\right)  $ is order isomorphic
to $L.$

We first show that $\Theta$ is onto. Let $\left[  \left(  \tilde{x}%
_{i}\right)  \right]  $ be an element in $SP_{w}\left(  X_{L}\right)  $. By
Lemma \ref{C3A} and the subsequent Remark, there exist a non-negative sequence
$\left(  c_{e}\right)  _{e\in L}$ with $\sum c_{e}=1$ and normalized spreading
models $\left(  \tilde{x}_{i}^{e}\right)  $ in $d\left(  w_{Te},1\right)  $
such that
\begin{equation}
\Vert\sum_{i}a_{i}\tilde{x}_{i}\Vert=\sum_{e\in L}c_{e}\| \sum_{i}a_{i}%
\tilde{x}_{i}^{e}\| . \label{CEQ1}%
\end{equation}
Since each $w_{Te}$ is 4-submuliplicative, according to Theorem \ref{ThmACL},
for each $e\in L,$ there exists $b_{e}>0$ such that%
\begin{equation}
b_{e}\|\sum_{i}a_{i}u_{i}^{e}\| \leq\| \sum_{i}a_{i}\tilde{x}_{i}^{e}\|
\leq5\| \sum_{i}a_{i}u_{i}^{e}\| . \label{CEQ2}%
\end{equation}
Let $I=\{e\in L:c_{e}>0\}$. If $I$ is infinite, write its elements in a
sequence $(e_{i})_{i=1}^{\infty}$. Since the sequence $(\vee_{i=1}^{n}%
e_{i})_{n=1}^{\infty}$ has no strictly increasing infinite subsequence, there
is a finite subset $J$ of $I$ such that $\vee_{e\in J}e\geq e^{\prime}$ for
all $e^{\prime}\in I$. If $I$ is finite, take $J=I$. Let $f=\vee_{e\in J}e$.
We claim that $\left(  \tilde{x}_{i}\right)  $ is equivalent to $(u_{i}^{f}).$
Observe that $e\leq f$ for all $e\in I$. Hence $Te\subseteq Tf$ and thus
$w_{Te}\leq w_{Tf}$. Therefore, $(u_{i}^{e})$ is $1$-dominated by $(u_{i}%
^{f})$. By (\ref{CEQ1}) and (\ref{CEQ2}),%
\begin{align*}
\Vert\sum_{i}  &  a_{i}\tilde{x}_{i}\Vert=\sum_{e\in L}c_{e}\Vert\sum_{i}%
a_{i}\tilde{x}_{i}^{e}\Vert=\sum_{e\in I}c_{e}\Vert\sum_{i}a_{i}\tilde{x}%
_{i}^{e}\Vert\\
&  \leq5\sum_{e\in I}c_{e}\Vert\sum_{i}a_{i}u_{i}^{e}\Vert\leq5\sum_{e\in
I}c_{e}\Vert\sum_{i}a_{i}u_{i}^{f}\Vert=5\Vert\sum_{i}a_{i}u_{i}^{f}\Vert.
\end{align*}
On the other hand, by Lemma \ref{C2}, $\bigoplus_{e\in J}\left(  u_{i}%
^{e}\right)  $ is equivalent to $(u_{i}^{f}).$ Using (\ref{CEQ1}) and
(\ref{CEQ2}) again,
\begin{align*}
\Vert\sum a_{i}\tilde{x}_{i}\Vert &  =\sum_{e\in I}c_{e}\| \sum_{i}a_{i}%
\tilde{x}_{i}^{e}\| \geq\sum_{e\in I}c_{e}b_{e}\| \sum_{i}a_{i}u_{i}^{e}\|\\
&  \geq\sum_{e\in J}c_{e}b_{e}\| \sum_{i}a_{i}u_{i}^{e}\| \geq\min_{e\in J}\{
c_{e}b_{e}\} \sum_{e\in J}\| \sum_{i}a_{i}u_{i}^{e}\|\\
&  \geq K\|\sum a_{i}u_{i}^{f}\| \text{ for some }K>0.
\end{align*}
This shows that $( \tilde{x}_{i}) $ is equivalent to $(u_{i}^{f}).$ Hence
$\Theta f=[(u_{i}^{f})]=[ ( \tilde{x}_{i})] $.

Next we show that
\begin{equation}
e_{1}\leq e_{2}\text{ }\Leftrightarrow\Theta e_{1}\leq\Theta e_{2}.
\label{orderpreserving}%
\end{equation}
If $e_{1}\leq e_{2},$ then $Te_{1}\subseteq Te_{2}$ and hence $w_{Te_{1}}\leq
w_{Te_{2}}.$ It follows that $[\left(  u_{i}^{e_{1}}\right)  ]\leq[\left(
u_{i}^{e_{2}}\right)  ]$. On the other hand, if $e_{1}\nleqslant e_{2},$ then
$T\left(  e_{1}\right)  \nsubseteq T\left(  e_{2}\right)  .$ Choose $p\in
T\left(  e_{1}\right)  \smallsetminus T\left(  e_{2}\right)  .$ By Corollary
\ref{A5}, $\left(  u_{i}^{e_{2}}\right)  $ does not dominate $\left(
v_{i}\right)  ,$ the unit vector basis of $d\left(  w_{p},1\right)  .$ But
obviously $\left(  u_{i}^{e_{1}}\right)  $ dominates $\left(  v_{i}\right)  $.
Hence $[\left(  u_{i}^{e_{1}}\right)  ]\nleq[\left(  u_{i}^{e_{2}}\right)  ].$
Note that (\ref{orderpreserving}) also implies that $\Theta$ is injective.
Hence $\Theta$ $:L\rightarrow SP_{w}\left(  X_{L}\right)  $ is an order isomorphism.
\end{proof}

\noindent\textbf{Remark.} The example given here is non-reflexive. Given a
countable semilattice $L$ without an infinite increasing sequence, the
$\ell^{p\text{ }}$ ($1 < p <\infty$) version of the space defined above, i.e.,
$X_{p}=\left(
{\textstyle\bigoplus_{e\in L}}
d\left(  w_{Te},p\right)  \right)  _{p},$ which is a reflexive space, has the
property that $SP_{w}\left(  X_{p}\right)  $ is order isomorphic to the
semilattice $\hat{L}=\left\{  a\right\}  \cup L,$ $a>e$ for all $e\in L.$ We
do not know how to obtain a reflexive example for general semilattices. In
fact, according to the authors of \cite{DOS}, it is not known if there is a
reflexive space $X$ such that $SP_{w}\left(  X\right)  $ is order isomorphic
to $\left(  \left\{  \left\{  1,2\right\}  ,\left\{  1\right\}  ,\left\{
2\right\}  \right\}  ,\subseteq\right)  .$

\end{document}